\documentclass[12pt]{amsart}

\usepackage[margin=1in]{geometry}
\usepackage{amsmath}
\usepackage{amssymb}
\usepackage{amsthm}
\usepackage{xypic,xcolor,verbatim}
\usepackage{hyperref}

\newcommand{\QQ}{\mathbb Q}
\newcommand{\CC}{\mathbb C}
\newcommand{\RR}{\mathbb R}
\newcommand{\PP}{\mathbb P}
\newcommand{\QQb}{\overline{\mathbb Q}}
\newcommand{\rat}{\dashrightarrow}

\renewcommand{\AA}{\mathbb A}

\theoremstyle{plain}
  
  \newtheorem{proposition}[equation]{Proposition}
  \newtheorem{theorem}[equation]{Theorem}
   
  \newtheorem{conjecture}[equation]{Conjecture}

\theoremstyle{definition}

\theoremstyle{remark}
  \newtheorem{remark}[equation]{Remark}

\newif\ifhascomments \hascommentstrue
\ifhascomments
\newcommand{\matt}[1]{{\color{red}[[\ensuremath{\spadesuit\spadesuit\spadesuit} #1]]}}
\else
  \newcommand{\matt}[1]{}
\fi

\newcounter{BKMcM-counter}

\title{A rational map with infinitely many points of distinct arithmetic degrees}
\author{John Lesieutre}
\address{John Lesieutre\\ Penn State University Mathematics Dept \\ 204 McAllister Building \\ University Park, State College, PA 16802 \\ USA}
\email{jdl@psu.edu}
\thanks{J.L. is partially supported by NSF grant DMS-1700898.}
\author{Matthew Satriano}
\address{Matthew Satriano\\ Department of Pure Mathematics\\ University of Waterloo\\ Waterloo, ON N2L 3G1\\ Canada}
\email{msatrian@uwaterloo.ca}
\thanks{M.S. is partially supported by NSERC grant RGPIN-2015-05631.}

\begin{document}

\maketitle

\begin{abstract}
Let $f\colon X\dasharrow X$ be a dominant rational self-map of a smooth projective variety defined over $\overline{\QQ}$. For each point $P\in X(\QQb)$ whose forward $f$-orbit is well-defined, Silverman introduced the arithmetic degree $\alpha_f(P)$, which measures the growth rate of the heights of the points $f^n(P)$. Kawaguchi and Silverman conjectured that $\alpha_f(P)$ is well-defined and that, as $P$ varies, the set of values obtained by $\alpha_f(P)$ is finite. Based on constructions of Bedford--Kim and McMullen, we give a counterexample to this conjecture when $X=\PP^4$.
\end{abstract}

\section{Introduction}
Let \(f \colon X \rat X\) be a dominant rational map of a smooth projective variety defined over \(\QQb\). We let $I_f$ denote the indeterminacy locus of $f$, and $X_f(\QQb)$ denote the set of $\QQb$-points of $X$ whose forward \(f\)-orbit is well-defined, i.e.~those $P\in X(\QQb)$ such that $f^n(P)\notin I_f$ for all $n\geq0$. To each point $P\in X_f(\QQb)$, Silverman \cite{arithmetic-degree-first-defined} introduced the following quantity which measures the arithmetic growth rate of $f^n(P)$. Fix an ample divisor $H$ on \(X\) and a logarithmic Weil height function \(h_H \colon X(\QQb) \to \RR\) for \(H\). Letting $h_H^+(P)=\max(h_H(P),1)$, consider the quantities
\[
\underline{\alpha}_f(P) = \liminf_{n \to \infty} h^+_H(f^n(P))^{1/n}, \qquad
\overline{\alpha}_f(P) = \limsup_{n \to \infty} h^+_H(f^n(P))^{1/n}.
\]
Kawaguchi and Silverman proved in \cite[Proposition 12]{dynamical-arithmetic-rat-maps} that these quantities are independent of the choice of ample divisor \(H\). When $\underline{\alpha}_f(P)=\overline{\alpha}_f(P)$, the \emph{arithmetic degree} \(\alpha_f(P)\) is defined to be the common limit. Kawaguchi and Silverman made the following conjecture and proved it in the case when $f$ is a morphism \cite[Theorem 3]{ks-jordan}.

\begin{conjecture}[{\cite[Conjecture 6abc]{dynamical-arithmetic-rat-maps}}]
\label{conj:KS-finiteness}
If $P\in X_f(\QQb)$, then the limit $\alpha_f(P)$ exists. Moreover, 
\[
\{\alpha_f(Q)\mid Q\in X_f(\QQb)\}
\]
is a finite set of algebraic integers.
\end{conjecture}

We prove the following result which gives a counterexample to Conjecture \ref{conj:KS-finiteness}.

\begin{theorem}
\label{thm:counter ex exists}
Let \(f \colon \PP^4 \rat \PP^4\) be the birational map defined by
\[
[X;Y;Z;A;B] \mapsto [XY + AX; YZ + BX; XZ; AX; BX].
\]
Then there exists a sequence of points $P_n\in X_f(\QQb)$ for which $\alpha_f(P_n)$ exists, and $\{\alpha_f(P_n)\}_n$ is an infinite set.
\end{theorem}

The strategy we use to prove Theorem \ref{thm:counter ex exists} is actually inspired by another conjecture of Kawaguchi and Silverman \cite[Conjecture 6d]{dynamical-arithmetic-rat-maps}, namely that if $P\in X_f(\QQb)$ and $P$ has Zariski dense orbit under $f$, then $\alpha_f(P)$ is equal to the first dynamical degree $\lambda_1(f)$.  Consider a family $\pi\colon X\to T$ and a dominant rational map $f\colon X\dasharrow X$ which preserves fibers and induces a dominant rational map \(f_t \colon X_t \rat X_t\) on every fiber.  For generic values of $t$, the first dynamical degrees $\lambda_1(f)$ and $\lambda_1(f_t)$ agree, but it is possible to have a countable union of subvarieties $\mathcal{T} \subset T$ such that $\lambda_1(f_t)<\lambda_1(f)$ for all $t\in\mathcal{T}$, and for which infinitely many distinct values arise as \(\lambda_1(f_t)\).  Suppose that for all $t\in\mathcal{T}$ we can find $P_t\in X_t(\QQb)$ whose forward orbit under $f_t$ is well-defined and Zariski dense in $X_t$. Then we would expect that $\alpha_f(P_t)=\alpha_{f_t}(P_t)=\lambda_1(f_t)$. Since the set of \(\lambda_1(f_t)\) is infinite, this would achieve infinitely many different values for $\alpha_f$.

There are a few issues one must handle in order to turn the above strategy into a counterexample to Conjecture \ref{conj:KS-finiteness}.  First, we must produce a suitable map \(f\),  and ensure that there are points $P_t$ with dense orbit under \(f_t\) and whose orbits avoid the indeterminacy of \(f\).
Second, one would expect that $\alpha_{f_t}(P_t)=\lambda_1(f_t)$ but this requires a proof. The easiest way to show this is to work in a case where \cite[Conjecture 6d]{dynamical-arithmetic-rat-maps} is already known to hold. For this reason, we consider a family of surface maps with $f$ birational and where $f_t$ extends to an automorphism of a birational model of $X_t$, so that we can appeal to \cite{surface-aut,dynamical-equals-arithmetic}, which proves that \(\alpha_f(P) = \lambda_1(f)\) in this case. We implement this strategy based on constructions of Bedford--Kim \cite{bk-periodicities} and McMullen \cite{mcmullen-dynamics-blow-up}.

\section{Proof of Theorem \ref{thm:counter ex exists}}

We begin by taking the strategy described in the introduction and codifying it as the following result.

\begin{proposition}
\label{prop:strategy-codified}
Let \(X\) be a smooth projective variety over \(\QQb\), and \(\pi \colon X \to T\) be a projective morphism of \(\QQb\)-varieties with two-dimensional fibers. Let \(f \colon X \rat X\) be a birational map defined over \(T\) and suppose there is an infinite sequence of parameters \(t_n\in T(\QQb)\) satisfying the following:
\begin{enumerate}
\item\label{strategy::birat-aut} for each \(n\), there exists a birational model \(\pi_{t_n} \colon \widetilde{X}_{t_n} \to X_{t_n}\) so that \(f_{t_n}\) extends to an automorphism $\widetilde{f}_{t_n} \colon \widetilde{X}_{t_n} \to \widetilde{X}_{t_n}$;
\item\label{strategy::dense-orbit} for each \(n\), there exists a \(\QQb\)-point \(P_n\) of \(X_{t_n}\), contained in the open set where \(\pi_{t_n}\)is an isomorphism, and with well-defined $f$-orbit that is Zariski dense in \(X_{t_n}\);
\item\label{strategy::infinite-lambda1s} the set of values \(\lambda_1(f_{t_n})\) is infinite.
\end{enumerate}
Then the set of values of \(\alpha_f(P_n)\) is infinite.
\end{proposition}
\begin{proof}
Fix an ample divisor $H$ on $X$. Since $H$ restricts to an ample on $X_t$, we see $\alpha_f(P)=\alpha_{f_t}(P)$ for all $P\in X_t(\QQb)$ 
such that the arithmetic degree is well-defined. So to complete the proof, it is enough to show $\alpha_{f_{t_n}}(P_n)=\lambda_1(f_{t_n})$.

Let \(\widetilde{P}_n\) be the unique point of \(\widetilde{X}_{t_n}\) with \(\pi_{t_n}(\widetilde{P}_n) = P_n\). 
We have $\alpha_{f_{t_n}}(P_n)=\alpha_{\widetilde{f}_{t_n}}(\widetilde{P}_n)$ by \cite[Theorem 3.4]{mss}, and $\lambda_1(f_{t_n})=\lambda_1(\widetilde{f}_{t_n})$ by \cite[Theorem 1.(2)]{rat-maps-normal-vars} and the discussion that follows it. Since $\widetilde{f}_{t_n}$ is a surface automorphism and $\widetilde{P}_n$ has dense orbit, \cite[Theorem 2c]{dynamical-equals-arithmetic} tells us that $\alpha_{\widetilde{f}_{t_n}}(\widetilde{P}_n)=\lambda_1(\widetilde{f}_{t_n})$, completing the proof.
\end{proof}

We next use a construction due in various guises to Bedford--Kim \cite{bk-periodicities} and McMullen \cite{mcmullen-dynamics-blow-up}. The relation between these two constructions is explained in the introduction of \cite{bk-dynamics-linear-fractional} as well as their remark on page 578. We 
collect the relevant facts from these papers in the following proposition.

\begin{proposition}
\label{prop:BK-McM-facts}
Let $X=\PP^2\times\AA^2$ and consider the map $f\colon X\dasharrow X$ whose fiber over $(a,b)\in\AA^2$ is given in affine coordinates by $f_{a,b}(x,y)=(y+a,\frac{y}{x}+b)$. There is a sequence $t_n=(a_n,b_n)\in\AA^2(\QQb)$ indexed by the integers $n\geq10$ with the following properties:
\begin{enumerate}
\item\label{BKMcM::Vn-lambda1} the first dynamical degree \(\lambda_1(f_{t_n})\) is given by the largest real root $\delta_n$ of the polynomial \(x^{n-2}(x^3-x-1)+ x^3 + x^2 - 1\);
\item\label{BKMcM::Vn-lambda1-limit} the numbers \(\delta_n\) increase monotonically in $n$ to \(\delta_\ast \approx 1.32472\dots\), the real root of \(x^3-x-1\);
\item\label{BKMcM::inv-cuspidal} there is an $f_{t_n}$-invariant cuspidal cubic curve \(C_{t_n} \subset \PP^2\) with cusp $q_{t_n}$ which is invariant under \(f_{t_n}\);
\item\label{BKMcM::birat-model-aut} there is a birational model \(\pi_{t_n}\colon\widetilde{X}_{t_n} \to \PP^2\) such that $f_{t_n}$ extends to an automorphism of $\widetilde{X}_{t_n}$; specifically, $\pi_{t_n}$ is a blow-up at $n$ points in the smooth locus of $C_{t_n}$;
\item\label{BKMcM::cusp} the point $q_{t_n}$ is not contained in the indeterminacy locus of \(f\);
\item\label{BKMcM::cusp-derivative} the derivative of \(f_{t_n}\) at $q_{t_n}$ is given in suitable coordinates by \(\left(
\begin{smallmatrix}
\delta_n^{-2} & 0 \\ 0 & \delta_n^{-3}
\end{smallmatrix}\right)\).
\end{enumerate}
\end{proposition}
\begin{proof}
First note that the indeterminacy locus of $f$ is $\{(1:0:0),(0:1:0),(0:0:1)\}\times\AA^2$.

Let $t_n=(a_n,b_n)$ be as on page 39 of \cite{mcmullen-dynamics-blow-up}. Let $p_1=(0:0:1)$, $p_2=(1:0:0)$, $p_3=(0:1:0)$, and $p_{4+i}=f_{t_n}^i(a_n:b_n:1)$ for $0\leq i\leq n-4$. By construction 
(see \S7), the $p_j$ 
lie in the smooth locus of a cuspidal cubic curve $C_{t_n}$, 
and letting $\pi_{t_n}\colon\widetilde{X}_{t_n} \to \PP^2$ be the blow-up at the $p_j$, the map $f_{t_n}$ extends to an automorphism $\widetilde{f}_{t_n}$ of $\widetilde{X}_{t_n}$.\footnote{For reference, McMullen denotes $C_{t_n}$, $\widetilde{X}_{t_n}$, and $\widetilde{f}_{t_n}$ by $X_n$, $S_n$, and $F_n$, respectively.} Moreover, $\widetilde{f}_{t_n}$ preserves an irreducible curve $Y_n\subset\widetilde{X}_{t_n}$ in the complete linear system of the anti-canonical bundle, and $C_{t_n}=\pi_{t_n}(Y_n)$. Since the cusp $q_{t_n}$ of $C_{t_n}$ is not a smooth point of the curve, $q_{t_n}$ is necessarily distinct from the $p_j$. In particular, $\pi_{t_n}$ is an isomorphism in a neighborhood of $q_{t_n}$. Since $Y_n$ is preserved by $\widetilde{f}_{t_n}$, we see $q_{t_n}$ is fixed by $\widetilde{f}_{t_n}$ and hence $f_{t_n}$. Finally, $q_{t_n}$ is not in the indeterminacy locus of $f$ as $q_{t_n}\notin\{p_1,p_2,p_3\}$. This handles statements (\ref{BKMcM::inv-cuspidal})--(\ref{BKMcM::cusp}).

By equation (9.1) of \cite{mcmullen-dynamics-blow-up}, the derivative of $f_{t_n}$ at $q_{t_n}$ has eigenvalues $\lambda_1(f_{t_n})^{-2}$ and $\lambda_1(f_{t_n})^{-3}$ so (\ref{BKMcM::cusp-derivative}) will follow upon showing $\lambda_1(f_{t_n})=\delta_n$ in (\ref{BKMcM::Vn-lambda1}).

Lastly, taking into account differences in notation explained in the remark on page 578 of \cite{bk-dynamics-linear-fractional}, we see $(a_n,b_n)$ belongs to the locus $V_{n-3}$ as defined in their equation (0.2). Statements (\ref{BKMcM::Vn-lambda1}) and (\ref{BKMcM::Vn-lambda1-limit}) then follow from \cite[Theorem 2]{bk-periodicities} by taking $\alpha=(a,0,1)$ and $\beta=(b,1,0)$.
\end{proof}

We now prove the main result.

\begin{proof}[{Proof of Theorem \ref{thm:counter ex exists}}]
We keep the notation of Proposition \ref{prop:BK-McM-facts}. By construction, $f$ gives a rational self-map of $\AA^4$ sending $(x,y,a,b)$ to $(y+a,\frac{y}{x}+b,a,b)$.  Taking projective coordinates \([X;Y;Z;A;B]\) on \(\PP^4\), our map extends to the birational map \(f : \PP^4 \rat \PP^4\) given by
\[
[X;Y;Z;A;B] \mapsto [XY + AX; YZ + BX; XZ; AX; BX].
\]

To prove the theorem, we apply Proposition \ref{prop:strategy-codified}. Condition (\ref{strategy::birat-aut}) of the proposition is met by virtue of Proposition \ref{prop:BK-McM-facts} (\ref{BKMcM::birat-model-aut}), and condition (\ref{strategy::infinite-lambda1s}) follows from Proposition \ref{prop:BK-McM-facts} (\ref{BKMcM::Vn-lambda1}) and (\ref{BKMcM::Vn-lambda1-limit}). So we need only find $P_n\in X_{t_n}(\QQb)$ whose forward orbit under $f$ is well-defined and Zariski dense in $X_{t_n}=\PP^2$, and for which \(P_n\) lies in the locus where \(\pi_{t_n}\) is an isomorphism.

Notice that by Proposition \ref{prop:BK-McM-facts} (\ref{BKMcM::Vn-lambda1}) and (\ref{BKMcM::Vn-lambda1-limit}), for each $n\geq10$ we have $\lambda_1(f_{t_n})=\delta_n\geq\delta_{10}>1$. From \cite[Theorem 1.(2)]{rat-maps-normal-vars}, we see $\lambda_1(\widetilde{f}_{t_n})=\lambda_1(f_{t_n})>1$. Theorem 1.1 (1) and Lemma 2.4 (1) of \cite{zhang-periodic-aut} then show there are only finitely many $\widetilde{f}_{t_n}$-periodic curves.

By Proposition \ref{prop:BK-McM-facts} (\ref{BKMcM::inv-cuspidal}) and (\ref{BKMcM::cusp-derivative}), $q_{t_n}$ is an attracting fixed point of $f_{t_n}$.  Fixing a metric \(d\) on \(\PP^2(\CC)\), we find that there exists an analytic open set \(U_n \subset \PP^2(\CC)\) containing \(q_{t_n}\) for which 
\(f_{t_n}(U_n) \subseteq U_n\) and for which there exists a constant \(C < 1\) so that for any \(u\) in \(U_n\), we have \(d(f_{t_n}(u),q_{t_n}) < C \, d(u,q_{t_n})\).  In particular, the set \(U_n\) does not contain any \(f_{t_n}\)-periodic point other than \(q_{t_n}\).
By (\ref{BKMcM::birat-model-aut}) and (\ref{BKMcM::cusp}), we can choose $U_n$ so that it avoids the indeterminacy locus of $f$ and such that $\pi_{t_n}\colon \widetilde{U}_n = \pi_{t_n}^{-1}(U_n)\to U_n$ is an isomorphism. 

Let $\widetilde{P}_n$ be any $\QQb$-point of $\widetilde{U}_n \setminus \bigcup_{\text{$C$ is $\widetilde{f}_{t_n}$-periodic}} C$, and \(P_n = \pi_{t_n}(\widetilde{P}_n)\).
Notice that the $f$-orbit of $P_n$ is contained in $U_n$, so the orbit is well-defined and contained in the locus over which $\pi_{t_n}$ is an isomorphism. By construction, $\widetilde{P}_n$ is not contained in any $\widetilde{f}_{t_n}$-periodic curve. At last, since \(\widetilde{P}_n\) lies in \(\widetilde{U}_n\), it is not $\widetilde{f}_{t_n}$-periodic.   Since \(\widetilde{P}_n\) is not periodic and does not lie on any \(\widetilde{f}_{t_n}\)-periodic curve, it must have Zariski dense orbit under \(\widetilde{f}_{t_n}\), so that \(P_n\) has dense orbit under \(f_{t_n}\).
\end{proof}

\begin{remark}
One can imagine various corrections to Conjecture \ref{conj:KS-finiteness} to circumvent the counter-example of Theorem \ref{thm:counter ex exists}. For example, one might ask that the map \(f : X \rat X\) does not preserve any fibration.  This does not seem sufficient, however. Indeed, the map \(g : \PP^5 \rat \PP^5\) defined by
\[
[X;Y;Z;A;B;C] \mapsto [XY + AX; YZ + BX; XZ; AX + CY; BX + CZ; C^2]
\]
does not appear to preserve a fibration, but the hyperplane \(C = 0\) is \(g\)-invariant, and the restriction of \(g\) to this hyperplane is the map \(f : \PP^4 \rat \PP^4\) of Theorem~\ref{thm:counter ex exists}. One might instead attempt to correct Conjecture \ref{conj:KS-finiteness} by requiring either:
\begin{enumerate}
\item There is no subvariety \(Z \subset X\) such that \(f\vert_Z\) preserves a fibration; or
\item The points \(P_n\) are of bounded degree over \(\QQ\).
\end{enumerate}
\end{remark}
\noindent We know of no counterexamples in these settings.

\section*{Acknowledgments}

We are grateful to Joseph Silverman for useful comments.

\bibliographystyle{alpha}
\bibliography{refs}
\end{document}